%% file: main.tex
\documentclass[oneside]{amsart}

\usepackage{amsmath,amsfonts,amssymb,graphics,amsthm}
\usepackage{cite}
\usepackage{hyperref}
\usepackage{enumitem}
\usepackage{bbm}

\usepackage{mathtools}

\usepackage{a4wide}

\newtheorem{theorem}{Theorem}[section]
\newtheorem{lemma}[theorem]{Lemma}
\newtheorem{proposition}[theorem]{Proposition}
\newtheorem{corollary}[theorem]{Corollary}

\theoremstyle{remark}
\newtheorem*{completion1*}{Proof of Theorem \ref{thm:main1}, assuming \eqref{eq:assump} holds}
\newtheorem*{completion2*}{Proof of Theorem \ref{thm:main2}, assuming \eqref{eq:assump} holds}
\newtheorem*{completion3*}{Proof of Theorem \ref{thm:main3}}
\newtheorem{remark}[theorem]{Remark}

\newcommand{\PP}{\mathbb{P}}

\DeclarePairedDelimiter\floor{\lfloor}{\rfloor}

\begin{document}
\author{Stephen Muirhead  \\
\tiny{University College London}  \\
\tiny \lowercase{s.muirhead@ucl.ac.uk}
 }

\title[Two-site localisation in the BTM with slowly varying traps]{Two-site localisation in the Bouchaud trap \\ model with slowly varying traps}

\begin{abstract}
We consider the Bouchaud trap model on the integers in the case that the trap distribution has a slowly varying tail at infinity. We prove that the model eventually localises on exactly two sites with overwhelming probability. This is a stronger form of localisation than has previously been established in the literature for the Bouchaud trap model on the integers in the case of regularly varying traps. Underlying this result is the fact that the sum of a sequence of i.i.d.\ random variables with a slowly varying tail is asymptotically dominated by the maximal term.
\end{abstract}

\thanks{This research was supported by a Graduate Research Scholarship from University College London and the Leverhulme Research Grant RPG-2012-608 held by Nadia Sidorova. We would like to thank Nadia Sidorova, David Croydon and two anonymous referees for their helpful suggestions.}
\date{\today}

\maketitle

\section{Introduction}
\label{sec:intro}
\input{introduction.tex}
\section{The BTM in an inhomogeneous trapping landscape}
\label{sec:two}
\input{two.tex}

\section{The trapping landscape}
\label{sec:trapping}
\input{trapping.tex}
%
\bibliography{paper}{}
\bibliographystyle{plain}
\end{document}

%% file: introduction.tex
\subsection{The Bouchaud trap model}
The Bouchaud trap model on the integers (BTM) is the continuous-time Markov chain $\{X_s\}_{s \ge 0}$ on $\mathbb{Z}$ with transition rates
\begin{align*} 
w_{z \to y} := \begin{cases}
\frac{1}{2 \sigma_z}, \,  & \text{if } |z-y| = 1, \\
0, & \text{otherwise} ,
\end{cases}
\end{align*}
where $\sigma:= \{\sigma_z\}_{z \in \mathbb{Z} }$ is a collection of independent identically distributed (i.i.d.) strictly-positive random variables known as the (random) \textit{trapping landscape}. The BTM describes a continuous-time random walk on $\mathbb{Z}$ in which the waiting time at each visit to a site $z$ is independent and distributed exponentially with mean $\sigma_z$, and the subsequent site is chosen uniformly at random from among the nearest neighbours of $z$. The BTM has its origins in the statistical physics literature, where it was proposed as a simple effective model for the dynamics of spin-glasses on certain time-scales (see, e.g., \cite{Bouchaud92}). For a general overview of the BTM see \cite{BenArous06}; for a discussion of links to other trap models see \cite{BenArous14}.

Although the BTM may be defined on arbitrary graphs by analogy with the above (see, e.g., \cite{BenArous06}), the BTM on the integers is of particular interest because it may exhibit \textit{localisation}, that is, at large times its probability mass function may concentrate on small subsets of the domain $\mathbb{Z}$. Intuitively, localisation occurs if the influence of the largest traps that the BTM has visited by a certain time dominates the influence of all other traps; naturally the existence and strength of this localisation depends on the upper tail of the trap distribution.  Previous work (see, e.g., \cite{FIN99, FIN02}) has studied localisation in the BTM in the case of integrable and regularly varying traps with index $\alpha \in (0,1)$. The present work continues this study in the case of slowly varying traps, which can be considered as the limiting case $\alpha = 0$.

\subsection{Localisation in the BTM} Let $\mathbf{P}$ denote the law of the trapping landscape $\sigma$, and define the c\`{a}dl\`{a}g, non-decreasing and unbounded function
\[ L(x) := \frac{1}{\mathbf{P}(\sigma_0 > x)} .  \]
Previous studies of localisation in the BTM have focused on the case that (i) $L$ is integrable at infinity; or (ii) $L$ is regularly varying with index $\alpha \in (0, 1)$ at infinity.\footnote{Recall that a function $L$ is said to be \textit{regularly varying} with index $\alpha > 0$ at infinity if $\lim_{u \to \infty} L(uv)/L(u) = v^\alpha$ for any $v > 0$.} In the first case, the BTM is known to satisfy a version of Donsker's invariance principle, i.e.\ the BTM, properly rescaled, converges to Brownian motion in the $t \to \infty$ limit. By contrast, it has been shown (in \cite{FIN99}) that in the second case the BTM, properly rescaled, converges to a spatially-subordinated Brownian motion now known as the FIN diffusion, the simplest of a more general class of spatially-subordinated Brownian motions introduced in \cite{BenArous13}. As a consequence, the BTM exhibits localisation in the sense that
\begin{align}
\label{eq:weakint}
\limsup_{s \to \infty} \, \sup_{z \in \mathbb{Z}} P_\sigma(X_s = z) \neq 0 \quad \mathbf{P}\text{-almost surely,} 
\end{align}
where $\{X_s\}_{s \ge 0}$ denotes the BTM in the trapping landscape $\sigma$, with $P_\sigma$ its law under the initial condition $X_0 = 0$. In other words, for almost all trapping landscapes there exist arbitrarily large times at which the BTM has non-negligible probability mass located at a single site.

In \cite{Bertin03}, it was suggested that a stronger form of localisation than \eqref{eq:weakint} should hold in the $\alpha \to 0$ limit, namely that at large times the probability mass of the BTM should eventually be carried by just two sites with overwhelming probability (with respect to the trapping landscape $\sigma$). Although the paper gave heuristic justifications, to the best of our knowledge this has not yet been rigorously established in the literature. 

The present paper takes up this suggestion, considering the case that $L$ satisfies the \textit{slow variation} property
\begin{align}
\label{eq:sv}
 \lim_{u \to \infty} \frac{L(uv)}{L(u)} \to 1 , \quad \text{for any } v > 0.
 \end{align}
Our main result (in Theorem \ref{thm:main1} below) is to confirm the prediction of \cite{Bertin03} that the BTM exhibits two-site localisation with overwhelming probability.  Indeed we make this statement more precise, describing the  two localisation sites explicitly as well as determining the limiting proportion of probability mass located at each site.

More generally, our results contribute to the growing understanding that the limiting properties of slowly varying trap models are described through extremal processes, in the same way that the limiting properties of regularly varying trap models are described through stable subordinators. In particular, it is well-established that slowly varying trap models may exhibit extremal ageing, especially on sub-exponential time-scales (see \cite{Gayrard12, Gun13} for extremal ageing in the BTM on the complete graph; \cite{BenArous12, Bovier13} for extremal ageing in the dynamics of spin-glass models). The scaling limit of certain slowly varying trap models have also been shown to converge to extremal-type processes (see, e.g., \cite{Croydon13, Croydon14}). Our localisation result is a natural analogue of this description for the BTM.

\subsection{Our results}
For the remainder of the paper we assume that the trapping landscape satisfies the slow variation property \eqref{eq:sv}. To describe our results explicitly, we define some notation. For each $t \ge 0$, define the level
$$ \ell_t := \min \{ s \ge 0 : s \, L(s) \ge t \} , $$
remarking that this is well-defined since $L$ is c\`{a}dl\`{a}g. Further, denote by $Z_t^{(1)}$ (respectively $Z_t^{(2)}$) the closest site to the origin on the positive (respectively negative) half-line where the trap value exceeds the level $\ell_t$, i.e.\
$$ Z_t^{(1)} := \min \{ z \in \mathbb{Z^+} : \sigma_z > \ell_t   \}  \quad \text{and} \quad Z_t^{(2)} := \max \{ z \in \mathbb{Z^-} : \sigma_z > \ell_t \}, $$
and let $\Gamma_t := \{Z_t^{(1)}, Z_t^{(2)} \}$, remarking that $\Gamma_t$ is $\mathbf{P}$-measurable.  Abbreviate $r_t := L(\ell_t)$, and note that $\ell_t, r_t \to \infty$ as $t \to \infty$.  Recall that $\{X_s\}_{s \ge 0}$ denotes the BTM in the trapping landscape $\sigma$, with $P_\sigma$ its law under the initial condition $X_0 = 0$. 

Our main result is to establish the property of two-site localisation.

\begin{theorem}[Two-site localisation in probability]
\label{thm:main1}
As $t \to \infty$,
\[  P_\sigma( X_t \in \Gamma_t ) \to 1 \quad \text{in } \mathbf{P} \text{-probability} .\]
\end{theorem}

Further, we determine the limiting proportion of probability mass located at each site and obtain the scaling properties of the localisation set. This allows us to establish the single-time scaling limit of the BTM. 

\begin{theorem}[Distribution between localisation sites]
\label{thm:main2}
For $i = 1,2$, as $t \to \infty$,
\[ P_\sigma(X_t = Z_t^{(i)} )  + \frac{|Z_t^{(i)}|}{\sum_{z \in \Gamma_t} |z| }  \to   1  \quad \text{in } \mathbf{P} \text{-probability.} \]
\end{theorem}

\begin{theorem}[Scaling of the localisation set]
\label{thm:main3}
As $t \to \infty$, 
\[ r_t^{-1} \left(Z_t^{(1)}, - Z_t^{(2)} \right) \Rightarrow (\mathcal{E}_1, \mathcal{E}_2)  \quad \text{in } \mathbf{P} \text{-law,} \]
where $\{\mathcal{E}_i\}_{i = 1,2}$ are independent exponential random variables with unit mean. 
\end{theorem}

\begin{corollary}
\label{cor:main4}
As $t \to \infty$, 
\[ \left(P_\sigma(X_t = Z_t^{(1)}), P_\sigma(X_t = Z_t^{(2)})   \right)  \Rightarrow  \left(\mathcal{U}, 1-\mathcal{U} \right)   \quad \text{in } \mathbf{P} \text{-law,} \]
where $\mathcal{U}$ is a uniform random variable on $[0, 1]$.
\end{corollary}

\begin{corollary}[Single-time scaling limit]
\label{cor:main5}
As $t \to \infty$,
$$ r^{-1}_t X_t \Rightarrow \,  y_1 \delta_{-x_1} + y_2 \delta_{x_2} \quad \text{in } \mathbf{P} \text{-law,} $$
where $\{x_i\}_{i = 1,2}$ are independent standard exponential random variables, $\delta_x$ is a Dirac measure at the point $x$, and each $i = 1,2$ satisfies $y_i := 1 - x_i / \sum_{j = 1,2} x_j $.
\end{corollary}

\begin{remark}
Theorems \ref{thm:main1}--\ref{thm:main3} collectively imply that, for a large fixed time $t$, the BTM is overwhelmingly likely to be located at either of the two (random) sites in $\Gamma_t$ and that the probability mass of the BTM will be asymptotically distributed between these two sites in (inverse) proportion to their distance to the origin. By Corollary 
\ref{cor:main4}, this implies that the probability mass of the BTM is distributed uniformly between the two localisation sites. Corollary \ref{cor:main5} summarises these results in a single-time scaling limit, but is considerably less precise, since it is not sensitive to the shape of the probability mass function of the BTM on fine distance scales. Note that Corollaries \ref{cor:main4} and \ref{cor:main5} follow trivially from the main results.

Remark that the two-site localisation result in Theorem \ref{thm:main1} holds in $\mathbf{P}$-probability. Quenched localisation results for the BTM with slowly varying traps -- i.e.\ localisation results that hold $\mathbf{P}$-almost surely, such as \eqref{eq:weakint} in the regularly varying case -- will be the subject of upcoming work.
\end{remark}

\begin{remark}
To gain some intuition about our results, recall the fundamental property of sequences of i.i.d.\ random variables with a slowly varying tail, namely that the sum is asymptotically dominated by the maximal term; this suggests that the dynamics of the BTM should be dominated by the effect of the deepest visited trap. Second, by standard properties of i.i.d.\ sequences, the spacing between the successive record deepest traps on the positive (respectively negative) half-line grows linearly with the distance from the origin. Hence, for the BTM to venture from the record deepest trap $z$ to an even deeper trap, it must travel a distance approximately $|z|$, and so will return to $z$ approximately $|z|$ times before doing so; such a displacement takes approximately a time $\sigma_z |z|$. Finally, standard extreme value estimates give $L(\sigma_z)$ as the correct scale for the location $|z|$ of the first trap of depth $\sigma_z$. Hence, this displacement takes approximately $\sigma_z L(\sigma_z)$ time. As such, we expect the BTM to be located on the first site $z$ that it visits such that $\sigma_z L(\sigma_z) > t$, i.e.\ the first site in $\Gamma_t$ that it hits. This is essentially the content of Theorems  \ref{thm:main1} and \ref{thm:main2}.
\end{remark}

\begin{remark}
Under the stronger assumption on $L$ that
\[ \lim_{u \to \infty} \frac{L(u / L(u))}{L(u)} \to 1, \]
the definition of $\Gamma_t$ can be considerably simplified by letting $\ell_t := t$. This is analogous to how simplified limit theorems are available under the above assumption in \cite{Croydon14, Kasahara86}. For simplicity, we choose not to prove this additional result here.
\end{remark}

\subsection{Outline of the proof}
\label{sec:outline}
The proof of our main results proceeds in the following steps:
\begin{enumerate}[leftmargin=1.2cm]
\item For a large fixed $t$, we show that the BTM is overwhelmingly likely to have hit the set $\Gamma_t$ before time $t$; 
\item Assuming that the event in (1) occurs, let $\bar y \in \Gamma_t$ denote the first site in $\Gamma_t$ hit by the BTM. We then show that the BTM is very unlikely to have exited a certain narrow region $I_t^{\bar y}$ around the site $\bar y$ by time $t$;  
\item Assuming that the events in (1) and (2) both occur, we use the equilibrium distribution of the BTM on an interval with periodic boundary conditions to show that the BTM is overwhelmingly likely to be located \textit{at} the site $\bar y$ at time $t$, establishing Theorem \ref{thm:main1}.
\item Remark that (1)--(3) above imply that the BTM is overwhelmingly likely to be located at the site in $\Gamma_t$ that it first hits. To finish the proof, we use simple properties of random walks and some basic extreme value theory to establish Theorems \ref{thm:main2} and \ref{thm:main3}. 
\end{enumerate}

The rest of the paper is organised as follows. In Section \ref{sec:two} we consider the probability law $P_\sigma$ of the BTM under the assumption that certain inhomogeneity properties of the trapping landscape $\sigma$ hold. Under this assumption we carry out the program outlined above, establishing the main results in Theorems \ref{thm:main1}--\ref{thm:main3}. In Section \ref{sec:trapping} we analyse the trapping landscape, showing that the inhomogeneity properties indeed do hold with overwhelming probability.

%% file: two.tex
In this section we complete the proof of Theorems \ref{thm:main1}--\ref{thm:main3} under the assumption that certain inhomogeneity properties of the trapping landscape $\sigma$ hold. In order to define these properties, we shall need an auxiliary function $h_t$ that tends to infinity (i.e.\ such that $h_t \to \infty$ as $t \to \infty$). We shall think of $h_t$ as being arbitrarily slowly growing, and indeed we shall require $h_t$ to satisfy $h_t^2 = o(r_t)$ as $t \to \infty$.\footnote{Note that we use $x_t = o(y_t)$ to mean that $\lim_{t \to \infty} x_t/y_t = 0$.} Further, define the quantities
\[  S_t := \sum_{Z_t^{(2)} < z < Z_t^{(1)}} \sigma_z, \qquad   d_t := \max_{z \in \Gamma_t} |z|  \qquad \text{and} \qquad m_t :=  \min_{z \in \Gamma_t}  \sigma_z , \]
and the $h$-dependent quantity
\[ \bar S_t := \sum_{i = 1,2} \ \sum_{ 1 \le |z - Z_t^{(i)}| < r_t/h_t } \sigma_z. \]
We may now define the inhomogeneity properties that we require, namely the ($\mathbf{P}$-measurable, $h$-dependent) events
\[
\mathcal{A}^h_t :=  \left\{ S_t d_t < \frac{t}{ h_t} \right\} , \quad \mathcal{B}^h_t := \left\{ m_t   >  \frac{t h^2_t}{r_t} \right\}  \quad \text{and} \quad \mathcal{C}^h_t := \left\{ \bar S_t  < \frac{\ell_t}{h_t} \right\} .\]
In Section \ref{sec:trapping}, we show that we can choose an $h_t$ growing sufficiently slowly such that, as $t \to \infty$,
\begin{align}
\label{eq:assump}  \mathbf{P} \left( \mathcal{A}^h_t, \mathcal{B}^h_t, \mathcal{C}^h_t \right) \to 1. 
\end{align}
For the remainder of this section we work under the assumption that \eqref{eq:assump} holds for a certain choice of $h_t$, showing how the main Theorems \ref{thm:main1}--\ref{thm:main3} follow from this assumption.

\subsection{Preliminary properties of random walks and Markov chains}
\label{sec:prop}

Here we collect some well-known results on random walks and Markov chains that will be useful in what follows. Let $D_n$ be the simple discrete-time random walk (SRW) on $\mathbb{Z}$ based at the origin. For a level $l > 0$ and a site $z \in \mathbb{Z}$ define the stopping time 
\[ a_l := \min \{ n : |D_n| \ge l \} ,\]
and the local time
\[ \mathcal{L}^l_z := | \{ n < a_l : D_n = z \}|. \]

\begin{proposition}[Bounds on local time for the SRW]
\label{prop:localtime}
As $l \to \infty$, both
\[ \frac{ \max_{z} \mathcal{L}^l_z }{l}   \quad \text{and} \quad \frac{\mathcal{L}^l_0}{l} \]
are bounded in probability above and away from zero.
\end{proposition}
\begin{proof}
These are simple consequences of invariance principles for random walk local times (see, e.g., \cite[Chapter 10]{Revesz90}). Indeed, these invariance principles actually imply the stronger result (see \cite[Theorem 7.6]{Borodin94}) that, as $l \to \infty$,
\[  \left( l^{-1} \mathcal{L}_{\floor{z l}}^l \right)_{z \in [-1, 1]} \stackrel{J_1}{\Rightarrow}   (\nu_z^1 )_{z \in [-1, 1] },  \]
where $\nu_z^1$ denotes the local time of Brownian motion at the point $z$ at the first hitting time of $\pm 1$, and $\stackrel{J_1}{\Rightarrow}$ denotes weak convergence in the Skorokhod space $D([0,1])$ of real-valued c\`{a}dl\`{a}g functions on $[0,1]$ equipped with the $J_1$ topology; see \cite{Whitt02} for a description.
\end{proof}

\begin{proposition}[Hitting probability for the SRW]
\label{prop:hittingprob}
For any $x \in \mathbb{Z}^+$ and $y \in \mathbb{Z}^-$, 
\[ \PP ( b_x < b_y )  = \frac{y}{x+y}, \]
where $b_z := \min \{n > 0: D_n = z \}$.
\end{proposition}
\begin{proof}
This well-known fact follows from the optional stopping theorem.
\end{proof}

\begin{proposition}[Monotonic convergence of a Markov chain to equilibrium]
\label{prop:MCconv}
Let $M_t$ be an irreducible, finite-state, time-homogeneous, continuous-time Markov chain, initialised at a state $0$, whose transition rates $w$ satisfy the detailed balance condition, i.e.\ there exists a non-negative vector $\pi$ such that
\[ \pi(x) w_{x \to y}= \pi(y) w_{y \to x} \] 
for each pair of states $x$ and $y$. Then $\pi$ is the unique equilibrium distribution for $M_t$ and satisfies, as $t \to \infty$,
\[ \PP(M_t = 0) \downarrow \pi(0) \quad \text{monotonically}. \]
\end{proposition}
\begin{proof}
This is a well-known result from continuous-time Markov chain theory. It can be proved by considering the spectral representation of $\PP(M_t = 0)$ in terms of the eigenvalues $\lambda_i$ and eigenfunctions $\varphi_i$ of the generator of $M_t$, i.e.\
\[ \PP(M_t = 0) = \sum_{i} e^{\lambda_{i} t} \varphi_i^2(0) , \]
recalling that the detailed balance condition ensures that each $\lambda_i$ and $\varphi_i$ is real. Since $\PP(M_t = 0)$ is bounded as $t \to \infty$, each $\lambda_i$ must satisfy $\lambda_i \le 0$, resulting in the monotonic convergence of $\PP(M_t = 0)$ to its equilibrium density.
\end{proof}

\subsection{Proving the main results} 

We are now in a position to carry out the program in Section \ref{sec:outline}, establishing Theorems \ref{thm:main1}--\ref{thm:main3} under the assumption that \eqref{eq:assump} holds. The properties $\mathcal{A}^h_t$, $\mathcal{B}^h_t$ and $\mathcal{C}^h_t$ will be used in steps (1)--(3) of the program respectively.\\

\noindent \textbf{Step 1: Hitting the localisation set.} Fix a scaling function $h_t$ such that \eqref{eq:assump} holds. For each trapping landscape $\sigma$ and time $t > 0$, consider the BTM $\{X_s\}_{s \ge 0}$ in the trapping landscape $\sigma$ and define the random time 
\[\tau^1_t := \inf \{s \ge 0: X_s \in \Gamma_t \}. \] 

\begin{proposition}
\label{prop:tau1}
Assume $\mathcal{A}^h_t$ holds. As $t \to \infty$,
\[ P_\sigma ( \tau_t^1 \le t ) \to 1. \]
\end{proposition}
\begin{proof}
Let $Q_z$ denote the discrete local time at $z$ of the geometric path induced by $\{X_s: s < \tau_t^1 \}$, and define 
\[ \bar \Gamma_t := \{ z \in  \mathbb{Z} : Z_t^{(2)} < z < Z_t^{(1)} \} . \] 
Considering $\tau^1_t$ as the sum of waiting times along the geometric path induced by $\{X_s : s< \tau_t^1\}$, we have that 
\[  \tau^1_t \stackrel{d}{=}   \sum_{z \in \bar \Gamma_t} \text{Gam} \left(  Q_z , \sigma_z \right) \prec \sum_{z \in \bar \Gamma_t} \text{Gam} \left( \max_z Q_z , \sigma_z \right), \]
  where $\prec$ denotes stochastic domination and each $\text{Gam}(n, \mu)$ is an independent gamma random variable with mean $n \mu$ and variance $n \mu^2$. Remark that, by the definition of $d_t$,
\begin{align}
\label{eq:loc}
  \frac{ \max_z Q_z }{d_t} \prec \frac{ \max_z L_z^{d_t} }{d_t} ,
  \end{align}
and recall that, by Proposition \ref{prop:localtime}, the right hand side of \eqref{eq:loc} is bounded above in probability.  Since $h_t \to \infty$,  this implies that, as $t \to \infty$,
\begin{align}
\label{eq:fourth}
P_\sigma \left( \tau_t^1 < \sum_{z \in \bar{\Gamma}_t} \text{Gam}(d_t h_t /2 , \sigma_z)  \right) \to 1.
\end{align}
Note that the factor of a half in the above equation is included purely for convenience in what follows. By Chebyshev's inequality,
\[ \PP \left( \text{Gam}(n, \mu) \ge 2 n \mu \right) \le n^{-1}, \]
and so, using the fact that
\[ \mathbb{P}\left(\sum_i Y_i \ge \sum_i y_i \right) \le \sum_i \mathbb{P} \left(Y_i \ge y_i \right) \]
for an arbitrary collection of random variables $\{Y_i\}_{i \in \mathbb{N}}$ and real numbers $\{y_i\}_{i \in \mathbb{N}}$, and also the fact that $|\bar \Gamma_t| < 2 d_t$ by definition, we have
\begin{align}
\label{eq:fifth}
P_\sigma \bigg( \sum_{z \in \bar \Gamma_t}   \text{Gam}( d_t h_t / 2, \sigma_z)  \ge  S_t d_t h_t \bigg)  &\le  \sum_{z \in \bar \Gamma_t} P_\sigma \bigg(   \text{Gam}( d_t h_t / 2, \sigma_z)  \ge  \sigma_z d_t h_t   \bigg)  \\
& \nonumber \le \frac{2 |\bar \Gamma_t|}{d_t h_t} < \frac{4 }{h_t} \to 0 \quad \text{as } t \to \infty. 
\end{align}
Since $S_t d_t h_t < t$ on $\mathcal{A}^h_t$, combining equations \eqref{eq:fourth} and \eqref{eq:fifth} yields the result.
\end{proof}

\noindent \textbf{Step 2: Confining to a narrow region.} Define the random site $\bar y  := X_{\tau_t^1} \in \Gamma_t$, a narrow region around $\bar y$ 
\[I^{\bar y}_t := \left\{ z \in \mathbb{Z} : |z - \bar y| <  r_t/h_t  \right\} , \]
and a second, strictly-later random time 
\[ \tau^2_t := \inf \{s>\tau^1_t : X_s \notin I_t^{\bar y} \}. \] 

\begin{proposition}
\label{prop:tau2}
Assume $\mathcal{B}_t^h$ holds. As $t \to \infty$,
\[ P_\sigma ( \tau^2_t > t ) \to 1. \]
\end{proposition}
\begin{proof}
Consider that 
\[ P_\sigma ( \tau^2_t > t ) \ge  P_\sigma ( \tau^2_t - \tau^1_t > t ),\]
so it is sufficient to prove that the latter probability converges to one. Let $Q_0$ denote the discrete local time at $\bar y$ of the geometric path induced by $\{X_s: \tau^t_1 \le s < \tau^t_2\}$. Following the same reasoning as in the proof of Proposition \ref{prop:tau1}, we have that
\[ \tau^2_t - \tau^1_t \stackrel{d}{=} \sum_{z \in I^{\bar y}_t} \text{Gam}(Q_z, \sigma_z) \succ \text{Gam}(Q_0, \sigma_{\bar y} ) \succ \text{Gam}(Q_0, m_t). \]
By Proposition \ref{prop:localtime}, as $t \to \infty$,
$$ \frac{ Q_0 }{ r_t / h_t} \stackrel{d}{=} \frac{ L_0^{r_t/h_t} }{ r_t / h_t }  $$
is bounded away from zero in probability, and so eventually
\begin{align}
\label{eq:sixth}
P_\sigma \left( \tau^2_t - \tau^1_t > \text{Gam} \left( \frac{2 r_t}{h^2_t}, m_t\right)   \right)  \to 1 .
\end{align}
Note that the factor of two here is again included for convenience in what follows. By Chebyshev's inequality,
\[ \PP \left( \text{Gam}(n, \mu) \le  n \mu / 2 \right) \le 4 n^{-1}, \]
and so, using the fact that $h_t^2 = o(r_t)$ as $t \to \infty$,
\begin{align}
\label{eq:seventh}
P_\sigma \left(\text{Gam} \left( \frac{2r_t}{h^2_t} , m_t \right) < \frac{ r_t m_t}{ h^2_t}   \right) < 2 h^2_t / r_t \to 0  \quad \text{as } t \to \infty.
\end{align}
Since $r_t m_t /h^2_t > t$ on $\mathcal{B}^h_t$, the result follows from combining \eqref{eq:sixth} and \eqref{eq:seventh}. 
\end{proof}

\noindent \textbf{Step 3: Two-site localisation.} Introduce a new random process $\{\hat{X}^t_s\}_{s \ge 0}$ on the same probability space as $\{X_s\}_{s \ge 0}$ which is: (i) coupled to $\{X_s\}_{s \ge 0}$ until time $\tau^1_t$; and (ii) thereafter evolves as the BTM on $I^{\bar y}_t$ with periodic boundary conditions. In other words,  $\{\hat{X}^t_{\tau^1_t + s} \}_{s \ge 0}$ is a continuous-time Markov chain on $I^{\bar y}_t$, based at $\bar y := X_{\tau^1_t}$ by definition, with transition rates 
\begin{align*} 
w_{z \to y} := \begin{cases}
\frac{1}{2 \sigma_z}, \,  & \text{if } z \stackrel{\ast}{\sim} y, \\
0, & \text{otherwise} ,
\end{cases}
\end{align*}
where $z \stackrel{\ast}{\sim} y$ denotes that $z$ and $y$ are either neighbours in $I^{\bar y}_t$ \textit{or} that $z$ and $y$ are the two end points of $I^{\bar y}_t$. 

\begin{proposition}
\label{prop:uhat}
Assume $\mathcal{C}_t^h$ holds. As $t \to \infty$,
\[ P_\sigma ( \hat{X}_t^t = \bar y \, | \, \tau^1_t \le t ) \to 1 .\]
\end{proposition}
\begin{proof}
Remark first that the BTM defined on any locally-finite graph satisfies the detailed balance condition. Hence we can apply Proposition \ref{prop:MCconv} to the irreducible, finite-state Markov chain $\{\hat{X}_{\tau_t^1 + s}^t \}_{s \ge 0}$. We conclude that, as $s \to \infty$,
\begin{align}
\label{eq:eighth}
P_\sigma( \hat{X}^t_{\tau_t^1 +s } = \bar y) \downarrow \pi(\bar y) \quad \text{monotonically,}
\end{align}
where $\pi$ is the equilibrium distribution of the BTM on $I_t^{\bar y}$ with periodic boundary conditions. We claim that this equilibrium distribution is proportional to the trapping landscape $\sigma$. To see why note that, by the definition of the BTM, $\pi$ satisfies
\[ (\Delta \boldsymbol\sigma ^{-1}) \pi = \Delta (\boldsymbol \sigma^{-1} \pi) = \mathbf{0}, \]
where $\Delta$ is the discrete Laplacian on $I_t^{\bar y}$ with periodic boundary conditions, $\boldsymbol \sigma$ denotes point-wise multiplication by $\sigma$, and $\mathbf{0}$ denotes the zero vector. Since the equilibrium distribution of $\Delta$ is uniform, the vector $\boldsymbol \sigma^{-1} \pi$ is also uniform, and the claim follows. 

As $\pi$ is proportional to $\sigma$, this implies that
\[ \pi(z) = \frac{\sigma_z}{\sigma_{\bar y}} \pi(\bar y) \le \frac{\sigma_z}{\sigma_{\bar y}} \]
for all $z \in I_t^{\bar y}$. Since, on the event $\mathcal{C}^h_t$, 
$$  \sum_{ z \in I^{\bar y}_t \setminus \{\bar y\} } \sigma_z \le \bar S_t  < \frac{\ell_t}{h_t} < \frac{\sigma_{\bar y}}{h_t} = o( \sigma_{\bar y} )  \quad \text{as } t \to \infty,$$
we have that, as $t \to \infty$,
\begin{align}
\label{eq:ninth}
\sum_{z \in  I_t^{\bar y} \setminus \{\bar y \} }  \pi(z) \to 0.
\end{align}
Combining equations \eqref{eq:eighth} and \eqref{eq:ninth} gives the result.
\end{proof}

\begin{completion1*}
We work on the event that each of $\mathcal{A}^h_t$, $\mathcal{B}^h_t$ and $\mathcal{C}^h_t$ holds, which is sufficient by \eqref{eq:assump}. Note that, by the definition of $\{\hat X_s^t\}_{s \ge 0}$,
\[  P_\sigma(\hat X_t^t \, | \, \tau_t^1 \le t < \tau_t^2 ) = P_\sigma(X_t \, | \, \tau_t^1 \le t < \tau_t^2 ) \]
Combining this with Propositions \ref{prop:tau1}-\ref{prop:uhat}, as $t \to \infty$,
\begin{align}
\label{eq:final}
P_\sigma(X_t = \bar y) \to 1, 
\end{align}
and we have the result.
\qed
\end{completion1*} 
\vspace{0.3cm}

\noindent \textbf{Step 4: Completion of the proof of Theorems \ref{thm:main2} and \ref{thm:main3}.}

\begin{completion2*}
Again we work on the event that each of $\mathcal{A}^h_t$, $\mathcal{B}^h_t$ and $\mathcal{C}^h_t$ holds, which is sufficient by \eqref{eq:assump}. Considering the BTM as a time-changed simple discrete-time random walk, it follows from Proposition \ref{prop:hittingprob} that
\[  P_\sigma ( \bar y = Z_t^{(1)} ) = \frac{ |Z_t^{(2)}| }{\sum_{z \in \Gamma_t} |z| }. \]
Combining with equation \eqref{eq:final} completes the proof.\qed
\end{completion2*}

\begin{completion3*}
Each $\sigma_z$ exceeds the level $l_t$ with probability
\[ \PP(\sigma_0 > l_t) = 1/L(l_t) = 1/r_t . \]
Hence for each $x,y > 0$, as $t \to \infty$,
\begin{align*}
\mathbf{P}(Z_t^{(1)} > x r_t , -Z_t^{(2)} > y r_t) = (1 - 1/r_t)^{\floor{x r_t} + \floor{ y r_t}} \sim e^{-x -y} ,
\end{align*}
which proves the result. \qed
\end{completion3*}

%% file: trapping.tex
In this section, we prove that the trapping landscape $\sigma$ is sufficiently inhomogeneous, in the sense that the events $\mathcal{A}^h_t, \mathcal{B}^h_t$ and $\mathcal{C}^h_t$ all hold eventually with overwhelming probability for a suitable choice of the slowly growing scaling function $h_t$. In other words, we prove that \eqref{eq:assump} holds.  This analysis relies crucially on the fundamental property of i.i.d.\ sequences of random variables with a slowly varying tail, namely that the sum is asymptotically dominated by the maximal term. 

\subsection{Specifying the scaling function}
Let us first specify an appropriate choice for $h_t$. The main condition we require is that $h_t \to \infty$ slowly enough so that, as $t \to \infty$,
\begin{align}
\label{eq:h}
 L(\ell_t / h^3_t) > L(\ell_t)(1 - 1/h_t)    \quad \text{and} \quad  L(\ell_t h^3_t) < L(\ell_t)(1 + 1/h_t)  
 \end{align}
eventually, remarking that such a choice is possible by the slow variation property \eqref{eq:sv}. For completeness, we construct an explicit scaling function $h_t$ satisfying \eqref{eq:h}. Define an arbitrary, positive, increasing sequence $c := (c_i)_{i \in \mathbb{N}} \uparrow \infty$, and denote, for each $u > 0$,
\[ f_t(u) := L( \ell_t u) / L( \ell_t) . \]
By the slow variation property \eqref{eq:sv}, for each $u$ we know that $f_t(u) \to 1$ as $t \to \infty$. This means that, for each $i \in \mathbb{N}$, there exists a $t_i > 0$ such that
\[ f_t(c^{-3}_i) > 1 - 1/c_i \quad \text{and} \quad  f_t(c^3_i) < 1 + 1/c_i \quad \text{for all } t \ge t_i . \]
So we can simply define $h_t$, with increments only on the set $\{t_i\}_{i \in \mathbb{N}}$, satisfying $h_{t_i} := c_i$; it is easy to check that $h_t$ satisfies equation \eqref{eq:h}. 

Recall also that we imposed the condition that $h^2_t = o(r_t)$ in Section \ref{sec:two}. To construct a scaling function $h_t$ that satisfies these two conditions simultaneously, simply take the minimum of scaling functions that satisfy each separately.

\subsection{Sequences of slowly varying random variables}
\label{sec:seq}
We begin our analysis of the trapping landscape by stating general properties of sequences of i.i.d.\ random variables with common distribution $\sigma_0$; let $Y := \{Y_n \}_{n \in \mathbb{N} }$ be such a sequence. Further, let $M := (M_n)_{n \ge 0}$ and $S := (S_n)_{n \ge 0}$ be,  respectively, the extremal and sum processes for the sequence $Y$, i.e.\
\[ M_n := \max \{ Y_i :  i \le \floor{n}  \} \quad \text{and} \quad  S_n := \sum_{i \le \floor{n}} Y_i . \]
The key to our analysis of the trapping landscape is the fact that the extremal and sum processes for $Y$ have scaling limits that coincide.

\begin{proposition}[Functional limit theorems for the extremal and sum process; see {\cite{Kasahara86, Lamperti64}}]
\label{prop:j1conv} As $n \to \infty$,
\begin{align}
\label{eq:j1}
 \left(\frac{1}{n} L (S_{n t} ) \right)_{t\geq 0} \stackrel{J_1}{\Rightarrow} \left(m_t \right)_{t\geq 0} \quad \text{and} \quad   \left(\frac{1}{n} L (M_{n t} ) \right)_{t\ge 0} \stackrel{J_1}{\Rightarrow} \left(m_t \right)_{t\ge 0},
 \end{align}
where $m := (m_t)_{t \ge 0}$ denotes the extremal process 
\[ m_t := \max \{ v_i : 0 \le x_i \le t\} \]
for the set $\mathcal{T} := (x_i, v_i)_{i \in \mathbb{N}}$, an inhomogeneous Poisson point process on $\mathbb{R}^+ \times \mathbb{R}^+$ with intensity measure $x^{-2} d x \, d v$, and $\stackrel{J_1}{\Rightarrow}$ denotes weak convergence in the Skorokhod space $D(\mathbb{R}^+)$ equipped with the $J_1$ topology; see \cite{Whitt02} for a description. Further, the convergence in equation \eqref{eq:j1} occurs jointly, in the sense that
\begin{align}
\label{eq:joint}
\left(\frac{1}{n} L(S_{nt}) - \frac{1}{n} L (M_{n t} ) \right)_{t\ge 0} \stackrel{J_1}{\Rightarrow} \left(0 \right)_{t\ge 0}  .
\end{align} 
\end{proposition}

The first statement in equation \eqref{eq:j1} is the main result of \cite{Kasahara86}; the second statement may be derived by applying \cite[Theorems 2.1, 3.2]{Lamperti64} to the sequence $Y$ (see \cite[Proposition 2.2]{Croydon14} for details). To establish the joint convergence in equation \eqref{eq:joint} we shall need the following two additional lemmas.

\begin{lemma}[Monotonicity implies joint convergence]
\label{lem:mono}
Let $\{X_n\}_{n \in \mathbb{N}}$ and $\{Y_n\}_{n \in \mathbb{N}}$ be sequences of random variable such that, as $n \to \infty$,
\[ X_n \Rightarrow Z \quad \text{and} \quad Y_n \Rightarrow Z  \quad \text{in law}\]
for some limiting random variable $Z$. Assume further that $X_n \ge Y_n$ for each $n$. Then, as $n \to \infty$,
\[ X_n - Y_n \Rightarrow 0 \quad \text{in law}.\]
\end{lemma}
\begin{proof}
For each $n \in \mathbb{N}$, $y \in \mathbb{R}$ and $\varepsilon > 0$ we have
\begin{align*} \mathbb{P}(Y_n > y) & = \mathbb{P}(Y_n > y, \, X_n - Y_n \ge \varepsilon) + \mathbb{P}(Y_n > y, \, |X_n - Y_n| < \varepsilon) \\
& \le \mathbb{P}(X_n > y + \varepsilon, \, X_n - Y_n \ge \varepsilon) + \mathbb{P}(X_n > y, \, |X_n - Y_n| < \varepsilon) \\
&= \mathbb{P}(X_n > y + \varepsilon) +  \mathbb{P}(X_n \in (y, y+\varepsilon], \,  |X_n - Y_n| < \varepsilon),
\end{align*}
and so
\[\mathbb{P}( X_n \in (y, y+\varepsilon], \,  |X_n - Y_n| < \varepsilon) \ge  \mathbb{P}(Y_n > y) -  \mathbb{P}(X_n > y + \varepsilon) \stackrel{n \to \infty}{\to} \mathbb{P}(Z \in (y, y + \varepsilon]) .\]
To complete the proof note that, for arbitrary $C > 0$, we can cover $(-C, C]$ with a finite number of disjoint regions $(y_i, y_i + \varepsilon]$. Summing over these, we have that, for each $C > 0$,
\[\liminf_{n \to \infty} \mathbb{P}( |X_n - Y_n| < \varepsilon) \ge \liminf_{n \to \infty} \mathbb{P}( X_n \in (-C, C], \,  |X_n - Y_n| < \varepsilon) \ge   \mathbb{P}(Z \in (-C, C]) .\]
Taking $C \to \infty$ establishes the result.
\end{proof}

\begin{lemma}
\label{lem:tight}
For $\varepsilon > \varepsilon' > 0$ and non-decreasing functions $x_t, y_t \to \infty$, there exists a $t'>0$ such that 
\[  \{ t > t':  L(x_t) >  (1 + \varepsilon) L(y_t)  \} \subseteq \{ t > t' : L(x_t - y_t) > (1 + \varepsilon') L(y_t)  . \]
\end{lemma}
\begin{proof}
By the slow variation property \eqref{eq:sv}, as $t \to \infty$ eventually
\[ (1 + \varepsilon ) L(y_t)> L(2 y_t)  . \]
Hence if $L(x_t) > (1 + \varepsilon) L(y_t)$, then $x_t > 2y_t$ eventually since $L$ is non-decreasing, and so $x_t - y_t > x_t/2$. This means that
\[ L(x_t - y_t) \ge L(x_t / 2) > (1 - \varepsilon'')L(x_t) \]
eventually for any $\varepsilon'' > 0$, again by the slow variation property \eqref{eq:sv}. The claim then follows by choosing $\varepsilon''$ such that $ (1 - \varepsilon'')(1 + \varepsilon) > (1 + \varepsilon') $.
\end{proof}

We can now establish the joint convergence in equation \eqref{eq:joint}.
\begin{proof}
By applying Lemma \ref{lem:mono} component-wise, the convergence in \eqref{eq:j1} implies that the finite-dimensional distributions of 
\[ \left(\frac{1}{n} L(S_{nt}) - \frac{1}{n} L (M_{n t}) \right)_{t \ge 0} \]
converge in law to the zero random vector; it remains to establish tightness in the topology of uniform convergence on compact sets. Using the criteria of \cite[Proposition VI.3.26]{Jacod87}, we need only check that, for arbitrary $0 < \delta < T$ and $\varepsilon > 0$,
\[ \lim_{C \to \infty} \lim_{n \to \infty} \mathbb{P}\left( \sup_{t \in [\delta, T]} \bigg| \frac{1}{n} L(S_{nt}) - \frac{1}{n} L (M_{n t}) \bigg| < C \right) = 0 \]
and
\[  \lim_{n \to \infty} \mathbb{P} \left(   \sup_{t \in [\delta, T]} \sup_{u,v \in [t, t+\delta] } \bigg| \left( \frac{1}{n} L(S_{nu}) - \frac{1}{n} L (S_{n v}) \right) - \left( \frac{1}{n} L(M_{nu}) - \frac{1}{n} L (M_{n v}) \right)  \bigg|   > \varepsilon \right) =  0\]
both hold. The first criterion is trivially satisfied by the convergence in \eqref{eq:j1}. For the second, since both $(n^{-1} L (S_{n t} ))_{t \ge 0}$ and $(n^{-1} L (M_{n t}))_{t \ge 0}$ converge in the $J_1$ topology to the pure-jump process $m_t$, it is sufficient to show that the (finite) set of non-negligible jumps of
\[ \left( n^{-1} L (S_{n t} ) \right)_{t \in [\delta, T]} \quad \text{and} \quad  \left(n^{-1} L (M_{n t} ) \right)_{t \in [\delta, T]} \]
are eventually matched exactly, i.e.\ that
\[  \lim_{n \to \infty} \mathbb{P} \left(   \sup_{t \in [\delta, T]} \bigg| \left( \frac{1}{n} L(S_{n t}) - \frac{1}{n} L (S_{n t^-}) \right) - \left( \frac{1}{n} L(M_{nt}) - \frac{1}{n} L (M_{n t^-}) \right)  \bigg|   > \varepsilon \right) = 0.\]
 Observe that, by the respective definitions of $M$ and $S$,
\[ M_{n t} \ge S_{n t} - S_{n t^-} \quad \text{and} \quad   S_{n t^-} \ge M_{n t^-}  . \]
Together with Lemma \ref{lem:tight} and the fact that $L$ is non-decreasing, this implies that, for any $\varepsilon > \varepsilon' > 0$, as $n \to \infty$ eventually we have the set inclusion
\[  \left\{ t \in [\delta, T] : \frac{1}{n} L (S_{n t} ) > (1 + \varepsilon)   \frac{1}{n} L (S_{n t^-} )  \right\} \subseteq  \left\{ t \in [\delta, T] : \frac{1}{n} L (M_{n t} ) > (1 + \varepsilon')   \frac{1}{n} L (M_{n t^-} )  \right\}  .\]
Since the jump sizes are bounded in probability, this implies that the non-negligible jumps of $(n^{-1} L (S_{n t} ))_{t \in [\delta, T]}$ are eventually matched exactly by non-negligible jumps of $(n^{-1} L (M_{n t}))_{t \in [\delta, T]}$. To complete the proof, note that if $t'>0$ denotes the time of the first non-negligible jump in $(n^{-1} L (M_{n t}))_{t \in [\delta, T]}$ that is unmatched by a jump in $(n^{-1} L (S_{n t} ))_{t \in [\delta, T]}$, then as $n \to \infty$ we would eventually have $M_{nt'} > S_{nt'}$, which is a contradiction.
\end{proof}

We now extract consequences of these scaling limits. 
For a level $l > 0$, define 
\[ n_l := \min \{n \in \mathbb{N} : M_n > l \}  \quad \text{and} \quad s_l := S_{n_l^-} = \sum_{i < n_l} Y_i  \]
to be respectively the index of the first exceedence of the level $l$ and the sum of all previous terms in the sequence. Further, for any $h > 0$, define
\[ \bar s^h_l :=  \sum_{ n : 1 \le |n - n_l| < L(l) / h }  Y_n . \]
Our aim is to analyse the four random variables $n_{\ell_t}$, $s_{\ell_t}$, $Y_{n_{\ell_t}}$ and $\bar s^{h_t}_{\ell_t}$. 
To assist in this analysis, we first need a preliminary asymptotic for $\ell_t$.

\begin{lemma}[Preliminary asymptotic for $\ell_t$] 
\label{lem:asymptl}
As $t \to \infty$, 
\[ \ell_t  \sim t/r_t  . \]
\end{lemma}
\begin{proof}
Recall that $\ell_t := \min \{s : s L(s) \ge t \}$ and so
\begin{align}
\label{eq:pre1}
\ell_t L(\ell_t^- ) \le t \le \ell_t L(\ell_t)  .
\end{align}
On the other hand, by the slow variation assumption \eqref{eq:sv}, as $u \to \infty$,
\[ L(u^-) \sim L(u) \]
which, combining with equation \eqref{eq:pre1}, gives the result.
\end{proof}

\begin{proposition}[Asymptotic law of the index of first exceedence]
\label{prop:dist}
As $l \to \infty$,
\[ \frac{ n_l }{ L(l) } \Rightarrow \mathcal{E}  \quad \text{in law},  \]
where $\mathcal{E}$ is an exponential random variable with unit mean.
\end{proposition}
\begin{proof} 
Each $Y_i$ exceeds the level $l$ with probability
\[ \PP(\sigma_0 > l) = 1/L(l) . \]
Hence, for each $x > 0$, as $l \to \infty$,
\begin{gather*} 
\PP( n_l > x L(l)) =  (1 - 1/L(l))^{\floor{x L(l)}} \sim e^{-x}  . \qedhere 
\end{gather*}
\end{proof}

\begin{proposition}[Upper bound on sum prior to first exceedence]
\label{prop:sum}
As $t \to \infty$,
\[ \PP \left( s_{\ell_t} < \frac{t}{2 r_t h^2_t} \right) \to 1 . \]
\end{proposition}
\begin{proof}
By the joint scaling limits for $M$ and $S$ in Proposition \ref{prop:j1conv}, as $l \to \infty$,
\[ L(s_l) / L(l) \]
converges in law to a certain $(0, 1)$-valued random variable. Hence, as $t \to \infty$,
\[ \PP \left(  L(s_{\ell_t}) < L(\ell_t)  (1  - 1 / h_t) \right) \to 1 . \]
Combining with the first statement in equation \eqref{eq:h} and the fact that $L$ is non-decreasing yields, as $t \to \infty$,
\[ \PP \left(  s_{\ell_t}  < \ell_t h^{-3}_t \right) \to 1 . \]
Finally, applying Lemma \ref{lem:asymptl} gives the result.
\end{proof}

\begin{proposition}[Lower bound on first exceedence]
\label{prop:ex}
As $t \to \infty$,
\[ \PP \left( Y_{n_{\ell_t}}  > \frac{ t h^2_t }{r_t} \right) \to 1 . \]
\end{proposition}
\begin{proof}
By the scaling limit for $M$ in Proposition \ref{prop:j1conv}, as $l \to \infty$,
\[   L(Y_{n_l}) / L(l) \]
converges in law to a certain $(1, \infty)$-valued random variable. Hence, as $t \to \infty$,
\[ \PP \left( L(Y_{n_{\ell_t}}) >  L(\ell_t) (1 + 1/h_t) \right) \to 1 . \]
Combining with the second statement in equation \eqref{eq:h}, the fact that $L$ is non-decreasing, and Lemma \ref{lem:asymptl} yields the result.
\end{proof}

\begin{proposition}[Bound on partial sum]
\label{prop:sum2}
As $t \to \infty$,
\[ \PP \bigg( \bar s_{\ell_t}^{h_t} < \frac{\ell_t}{h^3_t} \bigg) \to 1 . \]
\end{proposition}
\begin{proof}
We first claim that $\bar s_{\ell_t}^{h_t}$ is stochastically dominated by $S_{2r_t/h_t}$. This is since
\[  Y_i \stackrel{d}{=} \begin{cases}  
 Y_1 \big| \{ Y_1 \le \ell_t \}  \ \prec \ Y_1,  & \quad \text{if } i < n_{\ell_t}, \\
  Y_1,      & \quad \text{if } i > n_{\ell_t}  , \\
 \end{cases}  \]
 where $ Y_1 \big| \{ Y_1 \le \ell_t \} $ denotes the random variable $Y_1$ conditioned on the event that $\{ Y_1 \le \ell_t \} $, and moreover, for any $x > 0$ and $n \in \mathbb{N}$,
\[ | \{i: 1 \le |i-n| < x \}  | \le 2 x. \]
By the scaling limit for $S$ in Proposition \ref{prop:j1conv}, as $l \to \infty$,
\[ \frac{L(S_{2r_t/h_t})}{2r_t/h_t} \]
converges in law to a certain strictly-positive random variable. This implies that, as $t \to \infty$,
\[ \PP \left( L ( \bar s_{\ell_t}^{h_t} ) <  r_t (1 - 1/h_t) \right) = \PP \left( L ( \bar s_{\ell_t}^{h_t} ) <  L(\ell_t) (1 - 1/h_t) \right)\to 1 . \]
Combining with the first statement in equation \eqref{eq:h} and the fact that $L$ is non-decreasing yields the result. 
\end{proof}

\subsection{The trapping landscape is sufficiently inhomogeneous}
We are now in a position to prove that the events $\mathcal{A}^h_t, \mathcal{B}^h_t$ and $\mathcal{C}^h_t$ all hold eventually with overwhelming probability.

\begin{proposition}
\label{prop:eventa}
As $t \to \infty$,
\[ \mathbf{P}(\mathcal{A}^h_t) \to 1. \]
\end{proposition} 
\begin{proof}
Applying Proposition \ref{prop:dist} to the sequences $\{\sigma_z\}_{z \in \mathbb{N}^+}$ and $\{\sigma_z\}_{z \in \mathbb{N}^- \cup \{0\}}$ we have that, as $t \to \infty$,
\begin{align}
\label{eq:first}
\mathbf{P} \left( d_t < r_t h_t \right) \to 1  .
\end{align}
Similarly, applying Proposition \ref{prop:sum} to the same sequences, as $t \to \infty$,
\begin{align}
\label{eq:second}
 \mathbf{P} \left( S_t <  t / (r_t h^2_t) \right) \to 1   .
 \end{align}
Combining equations \eqref{eq:first} and \eqref{eq:second} yields the result.
\end{proof}

\begin{proposition}
\label{prop:eventb}
As $t \to \infty$,
\[ \mathbf{P}(\mathcal{B}^h_t) \to 1 . \]
\end{proposition} 
\begin{proof}
Similarly to the above, apply Proposition \ref{prop:ex} to the sequences $\{\sigma_z\}_{z \in \mathbb{N}^+}$ and $\{\sigma_z\}_{z \in \mathbb{N}^- \cup \{0\}}$.
\end{proof}

\begin{proposition}
\label{prop:eventc}
As $t \to \infty$,
\[ \mathbf{P}(\mathcal{C}^h_t) \to 1 .\]
\end{proposition} 
\begin{proof}
By Proposition \ref{prop:dist}, as $t \to \infty$, neither of the sets
\[ \{ z : |z-Z_t^{(i)}|  < r_t /h_t \},  \quad i = 1,2  \]
contains the origin with overwhelming probability. On this event, each of the sums
\[  \sum_{0 < |z-Z_t^{(i)}| < r_t/h_t } \sigma_z , \quad i = 1,2 \]
is distributed as an independent copy of the random variable $\bar s_{\ell_t}^{h_t}$ defined in Section \ref{sec:seq}. Applying Proposition \ref{prop:sum2} yields the result.
\end{proof}

%% file: main.bbl
\begin{thebibliography}{10}

\bibitem{BenArous13}
G.~{Ben Arous}, M.~Cabezas, J.~\v{C}ern\'{y}, and R.~Royfman.
\newblock Randomly trapped random walks.
\newblock {\em Ann. Probab. (to appear)}, 2014.

\bibitem{BenArous14}
G.~{Ben Arous} and A.~Fribergh.
\newblock Biased random walks on random graphs.
\newblock {\em arXiv:1406.5076}, 2014.

\bibitem{BenArous12}
G.~{Ben Arous} and O.~G{\"u}n.
\newblock Universality and extremal aging for dynamics of spin glasses on
  subexponential time scales.
\newblock {\em Comm. Pure Appl. Math.}, 65:77--127, 2012.

\bibitem{BenArous06}
G.~{Ben Arous} and J.~\v{C}ern\'{y}.
\newblock Dynamics of trap models.
\newblock {\em Math. Stat. Physics Lecture Notes -- Les Houches Summer School},
  83, 2006.

\bibitem{Bertin03}
E.~M. Bertin and J.~P. Bouchaud.
\newblock Subdiffusion and localization in the one-dimensional trap model.
\newblock {\em Phys. Rev.}, E 67:026128, 2003.

\bibitem{Borodin94}
A.~N. Borodin and A.~Ibragimov.
\newblock {\em Limit theorems for functionals of random walks}.
\newblock Proceedings of the Steklov Institute of Mathematics, 1994.

\bibitem{Bouchaud92}
J.~P. Bouchaud.
\newblock Weak ergodicity breaking and aging in disordered systems.
\newblock {\em J. Phys. I (France)}, 2:1705--1713, 1992.

\bibitem{Bovier13}
A.~Bovier, V.~Gayrard, and A.~\v{S}vejda.
\newblock Convergence to extremal processes in random environments and extremal
  ageing in {SK} models.
\newblock {\em Probab. Theory Related Fields}, 157:251--283, 2013.

\bibitem{Croydon13}
D.~Croydon, A.~Fribergh, and T.~Kumagai.
\newblock Biased random walk on critical {G}alton--{W}atson trees conditioned
  to survive.
\newblock {\em Probab. Theory Related Fields}, 157:453--507, 2013.

\bibitem{Croydon14}
D.~Croydon and S.~Muirhead.
\newblock Functional limit theorems for the {B}ouchaud trap model with slowly
  varying traps.
\newblock {\em Stoch. Process. Appl.}, 125(5):1980--2009, 2015.

\bibitem{FIN99}
L.~R.~G. Fontes, M.~Isopi, and C.~M. Newman.
\newblock Chaotic time dependence in a disordered spin system.
\newblock {\em Probab. Theory Related Fields}, 115(3):417--443, 1999.

\bibitem{FIN02}
L.~R.~G. Fontes, M.~Isopi, and C.~M. Newman.
\newblock Random walks with strongly inhomogeneous rates and singular
  diffusions: Convergence, localization and aging in one dimension.
\newblock {\em Ann. Probab.}, 30(2):579--604, 2002.

\bibitem{Gayrard12}
V.~Gayrard.
\newblock Convergence of clock process in random environments and aging in
  {B}ouchaud's asymmetric trap model on the complete graph.
\newblock {\em Electron. J. Probab.}, 17(58):1--33, 2012.

\bibitem{Gun13}
O.~G{\"u}n.
\newblock Extremal aging for trap models.
\newblock {\em arXiv:1312.1137}, 2013.

\bibitem{Jacod87}
J.~Jacod and A.~N. Shiryaev.
\newblock {\em Limit Theorems for Stochastic Processes}.
\newblock Springer, 2002.

\bibitem{Kasahara86}
Y.~Kasahara.
\newblock A limit theorem for sums of i.i.d. random variables with slowly
  varying tail probability.
\newblock {\em J. Math. Kyoto. Univ.}, 37:197--205, 1986.

\bibitem{Lamperti64}
J.~Lamperti.
\newblock On extreme order statistics.
\newblock {\em Ann. Math. Statist.}, 35:1726--1737, 1964.

\bibitem{Revesz90}
P.~R\'{e}v\'{e}sz.
\newblock {\em Random Walk in Random and Non-Random Environments}.
\newblock World Scientific, 1983.

\bibitem{Whitt02}
W.~Whitt.
\newblock {\em Stochastic-Process Limits}.
\newblock Springer, 2002.

\end{thebibliography}
